\documentclass[12pt]{amsart}
\usepackage{amssymb}

\usepackage{amsmath}
\usepackage{amscd}
\usepackage{a4}
\usepackage{bbm}
\usepackage{amscd}

\newtheorem{theorem}{Theorem}

\newtheorem{corollary}[theorem]{Corollary}

\newtheorem{definition}[theorem]{Definition}

\newtheorem{notation}[theorem]{Notation}

\newtheorem{proposition}[theorem]{Proposition}

\renewenvironment{proof}[1][Proof]{\textbf{#1.} }{\ \rule{0.5em}{0.5em}}

\begin{document}
\title{Turing Degrees and Automorphism Groups of Substructure Lattices}
\author{Rumen Dimitrov}
\address{Department of Mathematics, Western Illinois University, Macomb, IL
61455, USA}
\email{rd-dimitrov@wiu.edu}
\author{Valentina Harizanov}
\address{Department of Mathematics, George Washington University,
Washington, DC 20052, USA}
\email{harizanv@gwu.edu}
\author{Andrey Morozov}
\address{Sobolev Institute of Mathematics, Novosibirsk, 630090, Russia}
\email{morozov@math.nsc.ru}
\thanks{MSC-class: 03D45 (Primary) 03C57, 08A35 (Secondary)}
\thanks{The authors acknowledge partial support of the binational research
grant DMS-1101123 from the National Science Foundation. The second author
acknowledge support from the Simons Foundation Collaboration Grant and from
CCFF and Dean's Research Chair awards of the George Washington University.}
\thanks{Preliminary version of this paper appeared in the CiE conference
proceedings \cite{DHM1}.}

\begin{abstract}
The study of automorphisms of computable and other structures connects
computability theory with classical group theory. Among the noncomputable
countable structures, computably enumerable structures are one of the most
important objects of investigation in computable model theory. In this
paper, we focus on the lattice structure of computably enumerable
substructures of a given canonical computable structure. In particular, for
a Turing degree $\mathbf{d}$, we investigate the groups of $\mathbf{d}$%
-computable automorphisms of the lattice of $\mathbf{d}$-computably
enumerable vector spaces, of the interval Boolean algebra $\mathcal{B}_{\eta
}$ of the ordered set of rationals, and of the lattice of $\mathbf{d}$%
-computably enumerable subalgebras of $\mathcal{B}_{\eta }$. For these
groups we show that Turing reducibility can be used to substitute the
group-theoretic embedding. We also prove that the Turing degree of the
isomorphism types for these groups is the second Turing jump of $\mathbf{d}$%
, $\mathbf{d^{\prime \prime }}$.
\end{abstract}

\maketitle

\section{Automorphisms of effective structures}

Computable model theory investigates algorithmic content (effectiveness) of
notions, objects, and constructions in classical mathematics. In algebra
this investigation goes back to van der Waerden who in his \emph{Modern
Algebra} introduced an explicitly given field as one the elements of which
are uniquely represented by distinguishable symbols with which we can
perform the field operations algorithmically. The formalization of an
explicitly given field led to the notion of a computable structure, one of
the main objects of study in computable model theory. A structure is \emph{%
computable} if its domain is computable and its relations and functions are
uniformly computable. Further generalization and relativization of notion of
a computable structure led to computably enumerable (abbreviated by c.e.)
structures, as well as $\mathbf{d}$-computable and $\mathbf{d}$-c.e.\
computable structures for a given Turing degree $\mathbf{d}$. In
computability theory, Turing degrees are the most important measure of
relative difficulty of undecidable problems. All decidable problems have
Turing degree $\mathbf{0.}$ There are uncountably many Turing degrees and
they are partially ordered by Turing reducibility, forming an upper
semi-lattice.

Order relations are pervasive in mathematics. One of the most important such
relations is the embedding of mathematical structures. The structures we
focus on are the automophism groups of various lattices of algebraic
structures that are substructures of a large canonical computable structure.
The study of automorphisms of computable or computably enumerable structures
connects computability theory and classical group theory. The set of all
automorphisms of a computable structure forms a group under composition. We
are interested in matching the embeddability of natural subgroups of this
group with the Turing degree ordering. It is also natural to ask questions
about the complexity of the automorphism groups and its isomorphic copies
since isomorphisms do not necessarily preserve computability-theoretic
properties.

Our computability-theoretic notation is standard and as in \cite{So, Ro, EG,
FHM}. By $\mathcal{D}$ we denote the set of all Turing degrees, and by $%
\mathbf{d}=\deg (D)$ the Turing degree of a set $D$. Hence $\mathbf{0}=\deg
(\emptyset ).$ Turing jump operator is the main tool to obtain higher Turing
degrees. For a set $D,$ the jump $D^{\prime }$ is the halting set relative
to $D.$ Turing established that $\mathbf{d}^{\prime }=\deg (D^{\prime })>%
\mathbf{d}$. All computably enumerable sets have Turing degrees $\leq
\mathbf{0}^{\prime }$. By $\mathbf{d}^{\prime \prime }$ we denote $(\mathbf{d%
}^{\prime })^{\prime },$ and so on. We recall the following definition from
computability theory.

\begin{definition}
A nonempty set of Turing degrees, $I\subseteq \mathcal{D}$, is called a
\emph{Turing ideal} if:

$(1)$\text{ }$(\forall a\in I)(\forall b)[b\leq a\Rightarrow b\in I]$, and

$(2)$ $(\forall a,b\in I)[a\vee b\in I]$.
\end{definition}

\begin{notation}
Let $I$ be a Turing ideal. Let $\mathcal{M}$ be a computable structure. Then:

$(1)$ $Aut_{I}(\mathcal{M})$ is the set of all $\mathbf{d}$-computable
automorphisms of $\mathcal{M}$ for any $\mathbf{d}\in I$;

$\left( 2\right) $ If $I=\{\mathbf{s}:\mathbf{s}\leq \mathbf{d}\}$, then $%
Aut_{I}(\mathcal{M})$ is also denoted by $Aut_{\mathbf{d}}(\mathcal{M})$.
\end{notation}

When the structure $\mathcal{M}$ is $\omega $ with equality, then its
automorphism group $Aut(\mathcal{M})$ is usually denoted by $Sym(\omega )$,
the symmetric group of $\omega .$ Hence we have
\begin{equation*}
Sym_{\mathbf{d}}(\omega )=\{f\in Sym(\omega ):\deg (f)\leq \mathbf{d}\}\text{%
.}
\end{equation*}%
(See \cite{Morozov3, FHM, Morozov1, Morozov2, Morozov4, Morozov6, GHKMR} for
previous computability-theoretic results about $Aut(\mathcal{M})$.)\smallskip

The \emph{Turing degree spectrum} of a countable structure\ $\mathcal{A}$ is
\begin{equation*}
DgSp(\mathcal{A})=\{\deg (\mathcal{B}):\mathcal{B\cong A}\},
\end{equation*}%
where $\deg (\mathcal{B})$ is the Turing degree of the atomic diagram of $%
\mathcal{B}$. Knight \cite{JK} proved that the degree spectrum of any
structure is either a singleton or is upward closed. Only the degree
spectrum of a so-called automorphically trivial structure is a singleton,
and if the language is finite, that degree must be $\mathbf{0}$ (see~\cite%
{HM}). Automorphically trivial structures include all finite structures, and
also some special infinite structures, such as the complete graph on
countably many vertices. Jockusch and Richter (see \cite{Richter})\ defined
the \emph{Turing degree of the isomorphism type} of a structure, if it
exists, to be the least degree in its Turing degree spectrum. Richter \cite%
{Richter,R81} was first to systematically study such degrees. For these and
more recent results about these degrees see \cite{FHM}. In this paper, we
are especially interested in the following result by Morozov.

\begin{theorem}
(\cite{Morozov1}) The degree of the isomorphism type of the group $Sym_{%
\mathbf{d}}(\omega )$ is $\mathbf{d}^{\prime \prime }$.
\end{theorem}

\noindent We extend this and other computability-theoretic results by
Morozov to computable algebraic structures, in particular, vector spaces and
certain Boolean algebras. Preliminary version of this paper appeared in the
CiE conference proceedings \cite{DHM1}.

In the remainder of this section, we establish some general results about
lattices of subspaces of a computable vector space and subalgebras of the
interval Boolean algebra of the ordered set of rationals. There are two main
results in the paper. In Section 2, we establish exact correspondence
between embeddability of automorphism groups of substructures and the order
relation of the corresponding Turing degrees (see Theorem \ref{Main}). In
Section $3$, we compute the Turing degrees of the isomorphism types of the
corresponding automorphism groups (see Theorem \ref{Main2}).

Let $V_{\infty }$ be a canonical computable $\aleph _{0}$-dimensional vector
space over a computable field $F,$ which has a computable basis. We can
think of a presentation of $V_{\infty }$ in which the vectors in $V_{\infty
} $ are the (codes of) finitely non-zero $\omega $-sequences of elements of $%
F$. By $\mathcal{L}$ we denote the lattice of all subspaces of $V_{\infty }.$
For a Turing degree $\mathbf{d,}$ by $\mathcal{L}_{\mathbf{d}}(V_{\infty })$
we denote the following sublattice of $\mathcal{L}$:

\begin{equation*}
\mathcal{L}_{\mathbf{d}}(V_{\infty })=\{V\in \mathcal{L}:V\text{ is }\mathbf{%
d}\text{\textbf{-}computably enumerable}\}.
\end{equation*}%
Note that in the literature $\mathcal{L}_{\mathbf{0}}(V_{\infty })$ is
usually denoted by $\mathcal{L}(V_{\infty })$. About c.e. vector spaces see %
\cite{MN, DR, DHM}. Computable vector spaces and their subspaces have been
also studied in the context of reverse mathematics (see \cite{DHKLMM}).

Guichard \cite{Gu} established that there are countably many automorphisms
of $\mathcal{L}_{\mathbf{0}}(V_{\infty})$ by showing that every computable
automorphism is generated by a $1-1$ and onto computable semilinear
transformation of $V_{\infty }$. Recall that a map $\mu :V_{\infty
}\rightarrow V_{\infty }$ is called a \emph{semilinear transformation} of $%
V_{\infty }$ if there is an automorphism $\sigma $ of $F$ such that
\begin{equation*}
\mu (\alpha u+\beta v)=\sigma (\alpha )\mu (u)+\sigma (\beta )\mu (v)
\end{equation*}%
for every $u,v\in V_{\infty }$ and every $\alpha ,\beta \in F.$

By $GSL_{\mathbf{d}}$ we denote the group of $1-1$ and onto semilinear
transformations $\left\langle \mu ,\sigma \right\rangle $ such that $deg(\mu
)\leq \mathbf{d}$ and $deg(\sigma )\leq \mathbf{d}$.\smallskip

We will also consider the structure $\mathcal{B}_{\eta }$, which is the
interval Boolean algebra over the countable dense linear order without
endpoints. Let $\mathbb{Q}$ be a fixed standard copy of the rationals and
let $\eta $ be its order type. The elements of $\mathcal{B}_{\eta }$ are the
finite unions and intersections of left-closed right-open intervals$.$ For a
Turing ideal $I$, let $\mathcal{L}_{I}\mathcal{(B}_{\eta })$ be the lattice
of all subalgebras of $\mathcal{B}_{\eta }$ which are computably enumerable
(abbreviated c.e.) in some $\mathbf{d}\in I$.

\begin{notation}
Note that $\mathcal{L}_{\mathcal{D}}\mathcal{(B}_{\eta })$ is also denoted
by $\mathcal{L}at\mathcal{(B}_{\eta }),$ where $\mathcal{D}$ is the set of
all Turing degrees.

\begin{definition}
Note that $\mathcal{L}_{\mathbf{d}}\mathcal{(B}_{\eta })$ is the lattice of
all subalgebras of $\mathcal{B}_{\eta }$ which are c.e.\ in $\mathbf{d}$.

Note that $\mathcal{L}_{\mathbf{0}}\mathcal{(B}_{\eta })$ is also denoted by
$\mathcal{L(B}_{\eta })$.
\end{definition}
\end{notation}

In \cite{Gu}, Guichard proved that every element of $Aut(\mathcal{L}_{%
\mathbf{0}}(V_{\infty }))$ is generated by an element of $GSL_{\mathbf{0}}$.
This result can be relativized to an arbitrary Turing degree $\mathbf{d}$.

\begin{theorem}
(\cite{Gu}) Every $\Phi \in Aut(\mathcal{L}_{\mathbf{d}}(V_{\infty }))$ is
generated by some $\left\langle \mu ,\sigma \right\rangle \in $ $GSL_{%
\mathbf{d}}.$ Moreover, if $\Phi $ is also generated by some other $%
\left\langle \mu _{1},\sigma _{1}\right\rangle \in $ $GSL_{\mathbf{d}}$,
then there is $\gamma \in F$ such that
\begin{equation*}
\left( \forall v\in V_{\infty }\right) [\mu (v)=\gamma \mu _{1}(v)].
\end{equation*}
\end{theorem}

It follows from this result that every $\Phi \in Aut(\mathcal{L}%
_{I}(V_{\infty }))$ is generated by some $\left\langle \mu ,\sigma
\right\rangle \in $ $GSL_{I}.$

Every automorphism of $\mathcal{L}_{I}(V_{\infty })$ is defined on the
one-dimensional subspaces of $V_{\infty }$ and can be uniquely extended to
an automorphism of the entire lattice $\mathcal{L}$. Hence, we can identify
the automorphisms of $\mathcal{L}_{I}(V_{\infty })$ with their extensions to
automorphism of $\mathcal{L}$. We will prove that every automorphism of $%
\mathcal{L}$ is generated by a $\mathbf{d}$-computable semilinear
transformation and its restriction to $\mathcal{L}_{\mathbf{d}}(V_{\infty })$
is an automorphism of $\mathcal{L}_{\mathbf{d}}(V_{\infty }).$

\begin{proposition}
(a) $Aut(\mathcal{L}_{I}(V_{\infty }))=\bigcup\limits_{\mathbf{d}\in I}Aut(%
\mathcal{L}_{\mathbf{d}}(V_{\infty }))$

(b) $Aut(\mathcal{L})=\bigcup\limits_{\mathbf{d}\in \mathcal{D}}Aut(\mathcal{%
L}_{\mathbf{d}}(V_{\infty }))$
\end{proposition}

\begin{proof}
Part (b) follows immediately from (a).

To prove (a) note first that $\bigcup\limits_{\mathbf{d}\in I}Aut(\mathcal{L}%
_{\mathbf{d}}(V_{\infty }))\subseteq Aut(\mathcal{L}_{I}(V_{\infty }))$ by
the discussion above.

Now, suppose $\Phi \in Aut(\mathcal{L}_{I}(V_{\infty }))$. Let $\alpha
_{0},\alpha _{1},\alpha _{2},...$ be a fixed computable enumeration of the
elements of the field $F.$ Assume that $v_{0},v_{1},v_{2},...$ is a
computable enumeration of a computable basis of $V_{\infty }.$ Following\
Guichard's idea in \cite{Gu}, we define the following computable subspaces
of $V_{\infty }$:$\smallskip $

$V_{1}=span(\{v_{0},v_{2},v_{4},...\}),\smallskip $

$V_{2}=span(\{v_{1},v_{3},v_{5},...\}),\smallskip $

$V_{3}=span(\{v_{0}+v_{1},v_{2}+v_{3},v_{4}+v_{5},...\}),\smallskip $

$V_{4}=span(\{v_{1}+v_{2},v_{3}+v_{4},v_{5}+v_{6},...\}),\smallskip $

$V_{5}=span(\{v_{0}+\alpha _{0}v_{1},v_{2}+\alpha _{1}v_{3},v_{4}+\alpha
_{2}v_{5},...\}).\smallskip $

\noindent Suppose that $\Phi (V_{i})=W_{i}\in \mathcal{L}_{I}(V_{\infty })$
for $i=1,\ldots ,5$ and note that there are finitely many spaces involved,
there is a Turing degree $\mathbf{d}\in I$ such that $W_{i}\in \mathcal{L}_{%
\mathbf{d}}(V_{\infty })$. Using Guichard's method, we can prove that there
is a $\mathbf{d}$-computable semilinear transformation that induces an
automorphism $\Psi $ of $\mathcal{L}_{\mathbf{d}}(V_{\infty })$ and is such
that $\Psi =\Phi |_{\mathcal{L}_{\mathbf{d}}(V_{\infty })}$. Hence $\Phi $
as an automorphism of $\mathcal{L}_{I}(V_{\infty })$ is the unique extension
of $\Psi $.
\end{proof}

We can also prove that every automorphism of $Aut(\mathcal{L}_{\mathbf{d}}(%
\mathcal{B}_{\eta }))$ \ is generated by a $\mathbf{d}$-computable
automorphism of $\mathcal{B}_{\eta },$ and that every $\Phi \in Aut(\mathcal{%
L}_{I}(\mathcal{B}_{\eta }))$ is generated by a $\mathbf{d}$-computable
automorphism of $\mathcal{B}_{\eta }$ for some $\mathbf{d}\in I.$

\begin{proposition}
\label{Aut_BA_LBA}$Aut(\mathcal{L}_{I}(\mathcal{B}_{\eta }))\cong Aut_{I}(%
\mathcal{B}_{\eta })$
\end{proposition}

\begin{proof}
Suppose $\varphi :Aut_{I}(\mathcal{B}_{\eta })\rightarrow Aut(\mathcal{L}%
_{I}(\mathcal{B}_{\eta }))$ is an onto homomorphism that takes every member $%
\sigma \in Aut_{I}(\mathcal{B}_{\eta })$ to an automorphism $\varphi (\sigma
)$ of $\mathcal{L}_{I}(\mathcal{B}_{\eta })$ induced by $\sigma $. We will
show that $Ker(\varphi )=id.$ Assume that $\sigma $ is a nontrivial element
of $Ker(\varphi )$. It is easy to show that there exists an $a\neq 0$ such
that $\sigma (a)\neq \overline{a}$ and $\sigma (a)\cap a=0$. Then,
considering the image of the Boolean algebra $\left\{ 0,a,\overline{a}%
,1\right\} $ generated by $a,$ we obtain
\begin{equation*}
\varphi (\sigma )(\left\{ 0,a,\overline{a},1\right\} )=\left\{ \sigma
(0\},\sigma (a),\sigma (\overline{a}),\sigma (1)\right\} =\left\{ 0,\sigma
(a),\sigma (\overline{a}),1\right\} \neq \left\{ 0,a,\overline{a},1\right\} .
\end{equation*}

Hence $\varphi (\sigma )\neq id$. Therefore, $Ker(\varphi )$ is trivial and $%
\varphi $ is an isomorphism.
\end{proof}

\section{Turing reducibility and group embeddings for vector spaces and
Boolean algebras}

Morozov \cite{Morozov2} showed that the correspondence $\mathbf{a}%
\rightarrow Sym_{\mathbf{a}}(\omega )$ can be used to substitute Turing
reducibility with group-theoretic embedding. More precisely, he established
that
\begin{equation*}
Sym_{\mathbf{a}}(\omega )\hookrightarrow Sym_{\mathbf{b}}(\omega
)\Leftrightarrow \mathbf{a}\leq \mathbf{b}
\end{equation*}

\noindent for every pair $\mathbf{a}$, $\mathbf{b}$ of Turing degrees. It
follows from this result that
\begin{equation*}
Sym_{\mathbf{a}}(\omega )\cong Sym_{\mathbf{b}}(\omega )\Leftrightarrow
\mathbf{a}=\mathbf{b}.
\end{equation*}
Here, we establish analogous results for vector spaces and Boolean algebras.
In the proof of the next, main, theorem we will use the standard notation: $%
\left[ x,y\right] =x^{-1}y^{-1}xy$ and $x^{y}=y^{-1}xy.$

\bigskip

\begin{theorem}
\label{Main} For any pair of Turing ideals $I,J$ we have:

(a) $Aut(\mathcal{L}_{I}(V_{\infty }))\hookrightarrow Aut(\mathcal{L}%
_{J}(V_{\infty }))\Leftrightarrow I\subseteq J$

(b) $Aut_{I}(\mathcal{B}_{\eta })\hookrightarrow Aut_{J}(\mathcal{B}_{\eta
})\Leftrightarrow I\subseteq J$

(c) $Aut(\mathcal{L}_{I}(\mathcal{B}_{\eta }))\hookrightarrow Aut(\mathcal{L}%
_{J}(\mathcal{B}_{\eta }))\Leftrightarrow I\subseteq J$
\end{theorem}

\begin{proof}
We will prove (a) and (b) only. Statement (c) follows easily from (b) and
Proposition \ref{Aut_BA_LBA}.

(a) Assume $I\subseteq J.$ Then it is straightforward to show that $Aut(%
\mathcal{L}_{I}(V_{\infty }))\hookrightarrow Aut(\mathcal{L}_{J}(V_{\infty
}))$ .

Now, assume that $Aut(\mathcal{L}_{I}(V_{\infty }))\hookrightarrow Aut(%
\mathcal{L}_{J}(V_{\infty })).$ We will prove that $I\subseteq J\mathbf{.}$

Let $\mathbf{d}\in I$. As usual, let $\{e_{0},e_{1\text{ }},\ldots \}$ be a
fixed computable basis of $V_{\infty }$. For $\left\langle \mu _{i},\sigma
_{i}\right\rangle \in GSL_{\mathbf{d}}$, we define $\left\langle \mu
_{1},\sigma _{1}\right\rangle \thicksim \left\langle \mu _{2},\sigma
_{2}\right\rangle $ iff:

(1)\ $\sigma _{1}=\sigma _{2},$ and

(2)\ there is $\alpha \in F$ such that $\alpha \neq 0$ and $\left( \forall
v\in V_{\infty }\right) \left[ \mu _{1}(v)=\alpha \mu _{2}(v)\right] $.

Note that $Aut(\mathcal{L}_{\mathbf{d}}(V_{\infty }))\cong GSL_{\mathbf{d}%
}/\thicksim .$ We can define a group embedding $\delta :Sym_{\mathbf{d}%
}(\omega )\hookrightarrow GSL_{\mathbf{d}}/\thicksim $ as follows. For any $%
f\in Sym_{\mathbf{d}}(\omega )$, we let $\delta (f)$ be the $\thicksim $
-equivalence class of a linear transformation $\langle \widetilde{f}%
,id\rangle $ such that
\begin{equation}
\widetilde{f}(e_{i\text{ }})=e_{f(i)}.
\end{equation}%
Note that if $\delta (f_{1})=\delta (f_{2})$, then $\widetilde{f_{1}}=c%
\widetilde{f_{2}}$ for some $c\in F$, and thus
\begin{equation*}
\left( \forall i\in \omega \right) [e_{f_{1}(i)}=\widetilde{f_{1}}(e_{i})=c%
\widetilde{f_{2}}(e_{i\text{ }})=ce_{f_{2}(i)}].
\end{equation*}%
Since the vectors $e_{i}$, $i\in \omega $, are independent, we must have
\begin{equation*}
\left( \forall i\in \omega \right) [f_{1}(i)=f_{2}(i)].
\end{equation*}%
Therefore, there exists a map
\begin{equation*}
K:Sym_{\mathbf{d}}(\omega )\hookrightarrow GSL_{J}/\thicksim
\end{equation*}%
such that if $f\in Sym_{\mathbf{d}}(\omega ),$ then $K(f)$ is a $\mathbf{b}$%
-computable (for some $\mathbf{b}\in J)$ linear transformation of $V_{\infty
}$ modulo scalar multiplication.

We claim that if a set $A$ is c.e.\ in $\mathbf{d},$ then $A$ is c.e.\ in
some $\mathbf{c}\in J.$ Fix $A\subseteq \omega $ such that $A$ is c.e.\ in $%
\mathbf{d}$, and let $h:\omega \rightarrow \omega $ be a $\mathbf{d}$%
-computable enumeration of $A$, that is, $rng(h)=A.$ Fix a partition of the
natural numbers into uniformly computable infinite sets $R_{i}$ for $i\in
\mathbb{Z}$ with enumerations $R_{i}=\{c_{i}^{0}<c_{i}^{1}<\cdots \}.$ Let
the permutations $g_{0},g_{1},w,b\in Sym_{\mathbf{d}}(\omega )$ be defined
as follows:

$w(c_{i}^{j})=c_{i+1}^{j}$ for each $i\in Z$ and $j\in \omega $,\smallskip

$g_{0}=\prod\limits_{j\in \omega }(c_{0}^{2j},c_{0}^{2j+1})$,\smallskip

$g_{1}=\prod\limits_{j\in \omega }(c_{0}^{2j+1},c_{0}^{2j+2})$, and\smallskip

$b=\prod\limits_{n,t\in \omega \text{ }\wedge \text{ }%
h(t)=n}(c_{n}^{t},c_{n}^{t+1})$.\smallskip

We will also use the following abbreviation: $w^{n}=\underset{n\text{ times}}%
{\underbrace{w\cdots w}\text{.}}$ Then we have
\begin{equation*}
n\notin A\Leftrightarrow \left( \lbrack g_{0},b^{w^{n}}]=1\wedge \lbrack
g_{1},b^{w^{n}}]=1\right) \text{.}
\end{equation*}%
This is because $g_{0}$ and $b^{w^{n}}$ commute iff $n$ is not enumerated
into $A$ at an odd stage $t$, and, similarly, $g_{1}$ and $b^{w^{n}}$
commute iff $n$ is not enumerated into $A$ at an even stage $t.$ Let $%
\widetilde{g_{0}}=K(g_{0})$, $\widetilde{g_{1}}=K(g_{1})$, $\widetilde{w}%
=K(w)$, and $\widetilde{b}=K(b).$ Each $\widetilde{g_{0}}$, $\widetilde{g_{1}%
},$ $\widetilde{w}$ and $\widetilde{b}$ is computable in some (possibly
different degrees in $J$)$.$ Since $J$ is an ideal there is a fixed $\mathbf{%
c}\in J$ that computes all of them. Then

\begin{equation*}
n\notin A\Leftrightarrow \left( \lbrack K(g_{0}),K(b)^{\left( K(w)\right)
^{n}}]=1\wedge \lbrack K(g_{1}),K(b)^{\left( K(w)\right) ^{n}}]=1\right)
\end{equation*}%
\begin{equation*}
\Leftrightarrow \left( \lbrack \widetilde{g_{0}},\widetilde{b}^{\widetilde{w}%
^{n}}]_{/\thicksim }=1\wedge \lbrack \widetilde{g}_{1},\widetilde{b}^{%
\widetilde{w}^{n}}]_{/\thicksim }=1\right) \text{.}
\end{equation*}

We will now show that $[\widetilde{g_{0}},\widetilde{b}^{\widetilde{w}%
^{n}}]\nsim 1$ is c.e.\ relative to the fixed $\mathbf{c}\in J$. Let $\tau
_{n}~=_{def}~[\widetilde{g_{0}},\widetilde{b}^{\widetilde{w}^{n}}]$. Then

$\tau _{n}\nsim 1\Leftrightarrow \tau _{n}(e_{0})$ and $e_{0}$ are linearly
independent, or

$\qquad \qquad \ \ \left( \exists m\in \omega \right) \left( \exists \alpha
\neq 0\right) [\tau _{n}(e_{0})=\alpha e_{0}\wedge \tau _{n}(e_{m})\neq
\alpha e_{m}].$

\noindent Let $A$ have Turing degree $\mathbf{d}$. Then $A$ and $\overline{A}
$ are both c.e.\ in $\mathbf{d}$, and, therefore, $A$ is computable in $%
\mathbf{c}$. Hence $\mathbf{d}\leq \mathbf{c}$ and so $\mathbf{d\in }J$%
.\smallskip

(b) We will now prove that $Aut_{I}(\mathcal{B}_{\eta })\hookrightarrow
Aut_{J}(\mathcal{B}_{\eta })\Leftrightarrow I\subseteq J\mathbf{.}$

The proof is a corollary of the fact that
\begin{equation*}
Sym_{I}(\omega )\hookrightarrow Sym_{J}(\omega )\Leftrightarrow I\subseteq J%
\mathbf{.}
\end{equation*}

We first define an embedding
\begin{equation*}
H:Sym(\omega )\hookrightarrow Aut(B_{\eta }).
\end{equation*}

For any $p\in Sym(\omega ),$ define $\widetilde{p}:\mathbb{Q\rightarrow Q}$
as follows:%
\begin{equation*}
\widetilde{p}(x)=\left\{
\begin{array}{cc}
x & \text{if }x\ \text{is negative,} \\
p([x])+\{x\} & \text{if }x\text{ is non-negative},%
\end{array}%
\right.
\end{equation*}

where $[x]$, and $\{x\}$ are the integer and fractional parts of $x,$
respectively$.$

Then for any $p\in Sym(\omega )$, let $H(p)$ be an automorphism of $B_{\eta
} $ defined as follows. If $I\in $ $B_{\eta },$ then
\begin{equation*}
H(p)(I)=\widetilde{p}(I).
\end{equation*}

It is important to note that $H(p)$ is uniformly computable in $p.$

Now, suppose $Aut_{I}(\mathcal{B}_{\eta })\hookrightarrow Aut_{J}(\mathcal{B}%
_{\eta })$. Then
\begin{equation*}
Sym_{I}(\omega )\hookrightarrow _{H}Aut_{I}(\mathcal{B}_{\eta
})\hookrightarrow Aut_{J}(\mathcal{B}_{\eta })\hookrightarrow
_{ID}Sym_{J}(\omega ),\text{ }
\end{equation*}%
and so
\begin{equation*}
Sym_{I}(\omega )\hookrightarrow Sym_{J}(\omega ).
\end{equation*}

Thus, by Morozov's result, $I\subseteq J.$
\end{proof}

\begin{corollary}
For any pair of Turing degrees $\mathbf{a},\mathbf{b}$ we have%
\begin{eqnarray*}
Aut(\mathcal{L}_{\mathbf{a}}(V_{\infty })) &\hookrightarrow &Aut(\mathcal{L}%
_{\mathbf{b}}(V_{\infty }))\Leftrightarrow \mathbf{a}\leq \mathbf{b}; \\
Aut_{\mathbf{a}}(\mathcal{B}_{\eta }) &\hookrightarrow &Aut_{\mathbf{b}}(%
\mathcal{B}_{\eta })\Leftrightarrow \mathbf{a}\leq \mathbf{b.}
\end{eqnarray*}
\end{corollary}

\section{Turing degrees of the isomorphism types of automorphism groups}

In this section, we will determine the Turing degree spectra of both $GSL_{%
\mathbf{d}}$ and $Aut_{\mathbf{d}}(B_{\eta })$. More precisely, we will show
that each of them is the upper cone with the least element $\mathbf{d}%
^{\prime \prime }.$ For the statement of the main theorem we use terminology
and notation from the following definition.

\begin{definition}
A permutation $p$ on a set $M$ will be called:

$(i)$ $1_{\inf }2_{\inf }$ on $M$ if it is a product of infinitely many $1$%
-cycles and infinitely many $2$-cycles;

$(ii)$ $1_{\inf }2_{\text{fin}}$ on $M$ if it is a product of infinitely
many $1$-cycles and finitely many $2$-cycles.
\end{definition}

The main results about the degree spectra of $GSL_{\mathbf{0}}$ and $Aut_{%
\mathbf{0}}(B_{\eta })$ will use the following embeddability theorem, which
is interesting on its own.

\begin{theorem}
\label{GSL2}Let $G$ be an $X$-computable group, and let $H:Sym_{\mathbf{0}%
}(\omega )\hookrightarrow G$ be an embedding (of any complexity). Suppose
that for every $1_{\inf }2_{\inf }$ permutation $p\in Sym_{\mathbf{0}%
}(\omega )$, the image $H(p)$ is not a conjugate of the image of any $%
1_{\inf }2_{\text{fin}}$ permutation in $Sym_{\mathbf{0}}(\omega )$.

Then $\mathbf{0}^{\prime \prime }\leq \deg (X).$
\end{theorem}

\begin{proof}
Let $A$ be a $\Pi _{2}^{0}$-complete set and let $R(x,t)$ be a computable
predicate such that%
\begin{equation*}
n\in A\Leftrightarrow \left( \exists ^{\infty }t\right) R(n,t).
\end{equation*}%
We will prove that $A\leq _{T}X.$ Fix a partition of $\omega $ into
uniformly computable infinite sets $S_{i,j}$ for $i\in \mathbb{Z}$ and $j\in
\{1,2\}$ with enumerations $S_{i,j}=\{c_{i,j}^{0}<c_{i,j}^{1}<\cdots \}.$
The sets $S_{i,1}$ and $S_{i,2}$ will be referred to as the left and the
right parts of the $i$-th column $S_{i}=S_{i,1}\cup S_{i,2}.$ This reference
will be useful in the definitions of certain maps below. We can graphically
present this partition as follows:

\begin{equation*}
\cdots
\begin{tabular}{ccc}
\frame{%
\begin{tabular}{c}
. \\
. \\
$c_{-1,1}^{2}$ \\
$c_{-1,1}^{1}$ \\
$c_{-1,1}^{0}$%
\end{tabular}%
}\frame{%
\begin{tabular}{c}
. \\
. \\
$c_{-1,2}^{2}$ \\
$c_{-1,2}^{1}$ \\
$c_{-1,2}^{0}$%
\end{tabular}%
} & \frame{%
\begin{tabular}{c}
. \\
. \\
$c_{0,1}^{2}$ \\
$c_{0,1}^{1}$ \\
$c_{0,1}^{0}$%
\end{tabular}%
}\frame{%
\begin{tabular}{c}
. \\
. \\
$c_{0,2}^{2}$ \\
$c_{0,2}^{1}$ \\
$c_{0,2}^{0}$%
\end{tabular}%
} & \frame{%
\begin{tabular}{c}
. \\
. \\
$c_{1,1}^{2}$ \\
$c_{1,1}^{1}$ \\
$c_{1,1}^{0}$%
\end{tabular}%
}\frame{%
\begin{tabular}{c}
. \\
. \\
$c_{1,2}^{2}$ \\
$c_{1,2}^{1}$ \\
$c_{1,2}^{0}$%
\end{tabular}%
} \\
$\underbrace{\frame{%
\begin{tabular}{c}
$S_{-1,1}$%
\end{tabular}%
}\frame{%
\begin{tabular}{c}
$S_{-1,2}$%
\end{tabular}%
}}$ & $\underbrace{\frame{%
\begin{tabular}{c}
$S_{0,1}$%
\end{tabular}%
}\frame{%
\begin{tabular}{c}
$S_{0,2}$%
\end{tabular}%
}}$ & $\underbrace{\frame{%
\begin{tabular}{c}
$S_{1,1}$%
\end{tabular}%
}\frame{%
\begin{tabular}{c}
$S_{1,2}$%
\end{tabular}%
}}$ \\
$\text{column }S_{-1}$ & $\text{column }S_{0}$ & $\text{column }S_{1}$%
\end{tabular}%
\cdots
\end{equation*}

We will now define the following maps.\smallskip

\noindent (i)\ $w(c_{i+1,j}^{k})=_{def}c_{i,j}^{k}$ for each $i\in \mathbb{Z}%
,$ $k\in \omega $ and $j=1,2.$\smallskip

Clearly, the map $w$ is such that $w(S_{i+1,1})=S_{i,1}$ and $%
w(S_{i+1,2})=S_{i,2}.$ It maps the left (right) part of the $\left(
i+1\right) $-st column to the left (right) part of the $i$-th column for
each $i$.\smallskip

\noindent (ii)\ $p_{0}=_{def}\prod\limits_{k\in \omega
}(c_{0,1}^{k},c_{0,2}^{k})$

It is easy to see that the map $p_{0}$ switches the left and right parts of
the $0$-th column (i.e., $p_{0}(S_{0,1})=S_{0,2}$ and $%
p_{0}(S_{0,2})=S_{0,1} $), and is identity on all other elements of $\omega
. $ \smallskip

\noindent (iii)\ $p_{n}=_{def}p_{0}^{w^{n}}=w^{-n}p_{0}w^{n}$\smallskip

Note that the map $p_{n}$ switches the left and right parts of the $n$-th
column (i.e., $p_{n}(S_{n,1})=S_{n,2}$ and $p_{n}(S_{n,2})=S_{n,1}$), and is
identity on all other elements of $\omega $. \smallskip

\noindent (iv)\ $z(k)=_{def}\left\{
\begin{array}{ccc}
0 & \text{if} & k=0, \\
1 & \text{if} & k=2, \\
k-2 & \text{if} & k=2t\geq 4, \\
k+2 & \text{if} & k=2t+1.%
\end{array}%
\right. $

Note that the map $z$ is a permutation of $\omega $, which contains only one
infinite cycle and $\left( 0\right) $. \smallskip

\noindent (v)\ $\tau =_{def}(0,1)$\smallskip

For $k\in \mathbb{Z}$, we have
\begin{equation*}
\tau ^{z^{k}}=\left\{
\begin{array}{cc}
(0,2k)\text{ } & \text{if }k\geq 1, \\
(0,2\left| k\right| +1) & \text{if }k\leq 0,%
\end{array}%
\right.
\end{equation*}%
so%
\begin{equation}
\left( \forall n,m\in \omega \right) \left( \exists n_{1},m_{1}\in \mathbb{Z}%
\right) [\left( \tau ^{z^{n_{1}}}\right) ^{\tau ^{z^{m_{1}}}}=(n,m)].
\label{*}
\end{equation}

\noindent Note that property (\ref{*}) guarantees that any $1_{\inf }2_{%
\text{fin}}$ permutation on $\omega $ can be represented as a finite product
of the permutations $\tau $ and $z.$ \smallskip

\noindent (vi)\ We will now construct a permutation $b$ on $\omega $ with
the following properties:

$\qquad b\upharpoonright _{S_{n,1}}=id\upharpoonright _{S_{n,1}}$

$\qquad b\upharpoonright _{S_{n,2}}$ is $\left\{
\begin{array}{cc}
1_{\inf }2_{\inf }\text{ on }S_{n,2}\text{ } & \text{if }n\in A, \\
1_{\inf }2_{\text{fin}}\text{ on }S_{n,2} & \text{if }n\notin A\text{.}%
\end{array}%
\right. $

We will define $b$ in stages. At each stage $s$ we will have $%
E^{s}=_{def}dom(b^{s})=rng(b^{s}).$ \smallskip

\noindent \emph{Construction}\smallskip

\noindent \textit{Stage} $0$.

Let $b^{0}\upharpoonright S_{i}=_{def}id$ for $i\leq -1,$ and $%
E^{0}=\bigcup\limits_{i\leq -1}S_{i}.$

\noindent \textit{Stage} $s+1=\left\langle n,t\right\rangle $.\smallskip

Case 1. If $R(n,t)$, then find the least elements $p,q,r\in S_{n,2}$ such
that $p,q,r\notin E^{s}$. Let $b^{s+1}=b^{s}\cdot (p,q)$ and assume that $%
b^{s+1}(r)=r.$ Thus, we have $E^{s+1}=E^{s}\cup \left\{ p,q,r\right\} $ and $%
b^{s+1}\upharpoonright E^{s}=b^{s}$.\smallskip

Case 2. If $\lnot R(n,t)$, then find the least elements $p,q,r\in S_{n,2}$
such that $p,q,r\notin E^{s}$. Let $b^{s+1}\upharpoonright E^{s}=b^{s}$ and $%
b^{s+1}(p)=p,$ $b^{s+1}(q)=q,$ $b^{s+1}(r)=r.$ Then $E^{s+1}=E^{s}\cup
\left\{ p,q,r\right\} $.\smallskip

\emph{End of construction.}\smallskip

By construction, $dom(b)=rng(b)=\omega .$

It follows that\ if $n\in A,$ then $\left( \exists ^{\infty }t\right)
R(n,t), $ so Case 1 applies infinitely often for this $n$, and hence the map
$b$ switches infinitely many pairs in the right part of the $n$-th column.
Therefore, $b\upharpoonright S_{n,2}$ is $1_{\inf }2_{\inf }$ and $%
b\upharpoonright S_{n,1}=id$.

If $n\notin A,$ then $\left( \exists ^{<\infty }t\right) R(n,t)$, so Case 1
applies finitely often for this $n$, and hence the map $b$ switches only
finitely many pairs in the right part of the $n$-th column. Therefore, $%
b\upharpoonright S_{n,2}$ is $1_{\inf }2_{\text{fin}}$ and $b\upharpoonright
S_{n,1}=id$.

In both cases, the map $b^{p_{n}}$ reverses the action of $b$ on the left
part and the right part of the $n$-th column $S_{n}$, while for $k\neq n$,
we have $b^{p_{n}}\upharpoonright S_{k}=b\upharpoonright S_{k}$.\smallskip

Then $b\cdot b^{p_{n}}$ is $\left\{
\begin{tabular}{ll}
$1_{\inf }2_{\inf }$ on $S_{n}$ & if $n\in A,$ \\
$1_{\inf }2_{\text{fin}}$ on $S_{n}$ & if $n\notin A,$ \\
$id$ on $S_{k}$ & if $n\neq k.$%
\end{tabular}%
\right. $

Therefore, $b\cdot b^{p_{n}}$ is $\left\{
\begin{tabular}{ll}
$1_{\inf }2_{\inf }$ on $\omega $ & if $n\in A,$ \\
$1_{\inf }2_{\text{fin}}$ on $\omega $ & if $n\notin A.$%
\end{tabular}%
\right. $\smallskip

Finally, note that on $\omega ,$ every computable $1_{\inf }2_{\inf }$
permutation is the conjugate of a fixed computable $1_{\inf }2_{\inf }$
permutation and some other computable permutation. Therefore, assume that $f$
is a fixed computable $1_{\inf }2_{\inf }$ permutation such that for every $%
1_{\inf }2_{\inf }$ permutation $q\in Sym_{\mathbf{0}}(\omega )$:
\begin{equation*}
\left( \exists h\in Sym_{\mathbf{0}}(\omega )\right) [q=f^{h}].
\end{equation*}

Hence for every $n$, we have%
\begin{equation}
\begin{tabular}{lll}
$n\in A$ & $\Leftrightarrow $ & $b\cdot b^{p_{n}}$ is a $1_{\inf }2_{\inf }$
permutation on $\omega $ \\
& $\Leftrightarrow $ & $\left( \exists h\in Sym_{\mathbf{0}}(\omega )\right)
[b\cdot b^{p_{n}}=f^{h}]$ \\
& $\Leftrightarrow $ & $\left( \exists u\in H(Sym_{\mathbf{0}}(\omega
))\right) [H(b)\cdot H(b)^{H(p_{n})}=H(f)^{u}],$%
\end{tabular}
\label{A}
\end{equation}

\noindent and,%
\begin{equation}
\begin{tabular}{lll}
$n\notin A$ & $\Leftrightarrow $ & $b\cdot b^{p_{n}}$ is a $1_{\inf }2_{%
\text{fin}}$ permutation on $\omega $ \\
& $\Leftrightarrow $ & $b\cdot b^{p_{n}}=\prod\limits_{(i,j)\in F}\left(
\tau ^{z^{i}}\right) ^{\tau ^{z^{j}}}$ \\
& $\Leftrightarrow $ & $H(b)\cdot H(b)^{H(p_{n})}=\prod\limits_{(i,j)\in
F}\left( H(\tau )^{H(z)^{i}}\right) ^{H(\tau )^{H(z)^{j}}}.$ $\ \ \ \ \ \ \
\ $%
\end{tabular}
\label{B}
\end{equation}

\noindent The set $F$ in the last equality in (\ref{B}) denotes some finite
set of pairwise disjoint cycles. For the map $H:Sym_{\mathbf{0}}(\omega
)\hookrightarrow G,$ note that $H(p_{n})=H(w)^{-n}\cdot H(p_{0})\cdot
H(w)^{n}.$

We claim that the last equivalence in (\ref{A}) can be strengthened so that
we have:%
\begin{equation}
\begin{tabular}{lll}
$n\in A$ & $\Leftrightarrow $ & $\left( \exists u\in G\right) [H(b)\cdot
H(b)^{H(p_{n})}=H(f)^{u}].$%
\end{tabular}
\label{C}
\end{equation}

\noindent To prove ($\Rightarrow $) in \ref{C}), assume that $n\in A.$ Then
\begin{eqnarray*}
\left( \exists u\in H(Sym_{0}(\omega ))\right) [H(b)\cdot H(b)^{H(p_{n})}
&=&H(f)^{u}],\text{ hence} \\
\left( \exists u\in G\right) [H(b)\cdot H(b)^{H(p_{n})} &=&H(f)^{u}].
\end{eqnarray*}

\noindent We will prove ($\Longleftarrow $) by contradiction. Assume that
for some fixed $u\in G$ we have
\begin{equation*}
H(b)\cdot H(b)^{H(p_{n})}=H(f)^{u}\text{, but }n\notin A.
\end{equation*}%
Then, by (\ref{B}), we have the following: \smallskip

(i)\ $b\cdot b^{p_{n}}$ is a $1_{\inf }2_{\text{fin}}$ permutation on $%
\omega ,$ \smallskip

(ii)\ $H(b)\cdot H(b)^{H(p_{n})}$ is the image of the $1_{\inf }2_{\text{fin}%
} $ permutation $b\cdot b^{p_{n}},$ while \smallskip

(iii)\ $H(f)$ is the image of the $1_{\inf }2_{\inf }$ permutation $f.$%
\smallskip

\noindent This contradicts our assumption that the image under $H$ of the $%
1_{\inf }2_{\text{fin}}$ permutation $b\cdot b^{p_{n}}$ cannot be the
conjugate of the image of any $1_{\inf }2_{\inf }$ permutation, including $%
f. $

We will now show that $A$ is computable in the group $G$, and hence $A\leq
_{T}X.$ For a given $n\in \omega ,$ simultaneosly search for a finite set $F$
of pairwise disjoint cycles and an $u\in G$ such that either the last
equality in (\ref{B}) holds:$\smallskip $

\begin{equation*}
H(b)\cdot H(b)^{H(p_{n})}=\prod\limits_{(i,j)\in F}\left( H(\tau
)^{H(z)^{i}}\right) ^{H(\tau )^{H(z)^{j}}},
\end{equation*}%
$\smallskip $\noindent or the following equality from (\ref{C}) holds:$%
\smallskip $

\begin{equation*}
H(b)\cdot H(b)^{H(p_{n})}=H(f)^{u}.
\end{equation*}%
\noindent If the former search succeeds, then $n\notin A,$ while if the
latter search succeeds, then $n\in A.$
\end{proof}

We will now prove our main results about the degree spectra of automorphism
groups.

\begin{theorem}
The degree of the isomorphisms type of the group $GSL_{\mathbf{0}}$ is $%
\mathbf{0}^{\prime \prime }.$
\end{theorem}

\begin{proof}
Let $V=\{v_{0},v_{1\text{ }},\ldots \}$ be a computable basis of $V_{\infty
}.$ Define
\begin{equation*}
H:Sym_{\mathbf{0}}(\omega )\hookrightarrow GSL_{\mathbf{0}}
\end{equation*}%
so that for any $p\in Sym_{\mathbf{0}}(\omega )$ the image $%
H(p)=\left\langle L,id\right\rangle $ is a semilinear map such that
\begin{equation*}
L(v_{i})=v_{p(i)}\text{ for every }i\in \omega .
\end{equation*}%
We claim that under $H$, the image of a $1_{\inf }2_{\inf }$ permutation
from $Sym_{\mathbf{0}}(\omega )$ cannot be a conjugate of the image of a $%
1_{\inf }2_{\text{fin}}$ permutation from $Sym_{\mathbf{0}}(\omega )$. To
establish this fact, suppose that $\left\langle f,id\right\rangle $, $%
\left\langle f_{1},id\right\rangle \in GSL_{\mathbf{0}}$ are the images of
some $1_{\inf }2_{\inf }$ and $1_{\inf }2_{\text{fin}}$ computable
permutations on $\omega $, respectively. Suppose that $\left\langle
f,id\right\rangle $ and $\left\langle f_{1},id\right\rangle $ are
conjugates, and let $\left\langle h,\sigma \right\rangle \in GSL_{\mathbf{0}%
} $ be such that $\left\langle f,id\right\rangle ^{\left\langle h,\sigma
\right\rangle }=\left\langle f_{1},id\right\rangle .$ Note that the map $%
h:V_{\infty }\rightarrow V_{\infty }$ is $1-1$ and onto $V_{\infty }.$ To
simplify the notation, we will refer to the semilinear maps $\left\langle
f,id\right\rangle $, $\left\langle f_{1},id\right\rangle $, and $%
\left\langle h,\sigma \right\rangle $ simply as $f,f_{1}$, and $h,$
respectively.

We can view $f\upharpoonright V$ and $f_{1}\upharpoonright V$ as $1_{\inf
}2_{\inf }$ and $1_{\inf }2_{\text{fin}}$ permutations on $V,$ respectively.
We will prove that $f_{1}$ satisfies the property:
\begin{equation}
\left( \exists W\subset _{fin}V_{\infty }\right) \left( \forall v\in
V_{\infty }\right) [\left( v-f_{1}(v)\right) \in W].  \label{D}
\end{equation}%
\smallskip where, $W\subset _{fin}V_{\infty }$ stands for $W$ being a
finite-dimensional subspace of $V_{\infty }.$ To prove (\ref{D}), assume
that $B=\left\{ x_{1},\ldots ,x_{k},y_{1},\ldots ,y_{k}\right\} \subseteq V$
is such that $f_{1}\upharpoonright V=$ $\prod\limits_{i\leq k}\left(
x_{i},y_{i}\right) .$ Note that for every $v\in V_{\infty }$, there are $%
v_{1}\in span(V-B)$ and $v_{2}\in span(B)$ such that $v=v_{1}+v_{2}.$ Then
\begin{equation*}
f_{1}(v)=f_{1}(v_{1})+f_{1}(v_{2})=v_{1}+f_{1}(v_{2}),
\end{equation*}%
and so
\begin{equation*}
v-f_{1}(v)=v_{1}+v_{2}-v_{1}-f_{1}(v_{2})=v_{2}-f_{1}(v_{2})\in span(B)
\end{equation*}%
since $f_{1}(v_{2})\in span(B).$ Therefore, $W=span(B)$ is a
finite-dimensional subspace of $V_{\infty }$ for which property $(D)$ holds.

We will now prove that $f^{h}$ does not satisfy property (\ref{D}), which
will contradict the assumption that $f^{h}=f_{1}$. Thus, assume that $W$ is
a finite-dimensional subspace of $V_{\infty }$ such that
\begin{equation}
\left( \forall x\in V_{\infty }\right) [\left( x-f^{h}(x)\right) \in W].
\label{E}
\end{equation}%
Let $W_{1}=h(W)$ and note that $W_{1}$ is finite-dimensional$.$ Let $B_{1}$
be a finite subset of the basis $V$ such that%
\begin{equation*}
\left( \forall x\in W_{1}\right) [supp_{V}(x)\subseteq B_{1}],
\end{equation*}%
where $supp_{V}(x)$ denotes the support of $x$ with respect to the basis $V.$

We will now find $u_{1}\in V_{\infty }$ such that $u_{1}-f(u_{1})\notin
W_{1} $. Since $f\upharpoonright V$ is a $1_{\inf }2_{\inf }$ permutation on
$V,$ there are infinitely many pairs $\left( u,v\right) \in V\times V$ such
that
\begin{equation}
u\neq v,\text{ }f(u)=v\text{ and }f(v)=u.  \label{F}
\end{equation}%
Since $B_{1}$ is finite, we can also find $u_{1},v_{1}\in V-B_{1}$, which
have property (\ref{F}). Then:

(i) \ $u_{1}-f(u_{1})=u_{1}-v_{1}\neq 0,$ and

(ii) \ $u_{1}-f(u_{1})=\left( u_{1}-v_{1}\right) \notin span(B_{1})$ because
$B_{1}\cup \{u_{1},v_{1}\}\subseteq V$.

\noindent Since $W_{1}\subseteq span(B_{1})$, we have $u_{1}-f(u_{1})\notin
W_{1}.$

\noindent Therefore, $\left( h^{-1}(u_{1})-h^{-1}(f(u_{1}))\right) \notin
h^{-1}(W_{1})$,

\noindent and so $\left( h^{-1}(u_{1})-h^{-1}fhh^{-1}(u_{1})\right) \notin W.
$

\noindent If we let $x_{1}=h^{-1}(u_{1})$, we obtain
\begin{equation*}
x_{1}-f^{h}(x_{1})\notin W,
\end{equation*}%
which contradicts (\ref{E}).\smallskip

Thus, we constructed an embedding $H:Sym_{\mathbf{0}}(\omega
)\hookrightarrow GSL_{\mathbf{0}}$ such that the images of any $1_{\inf
}2_{\inf }$ and $1_{\inf }2_{\text{fin}}$ permutations from $Sym_{\mathbf{0}%
}(\omega )$ cannot be conjugates in $GSL_{\mathbf{0}}$. By Theorem \ref{GSL2}%
, we conclude that $\mathbf{\emptyset }^{\prime \prime }$ is computable in
any copy of $GSL_{\mathbf{0}}.$ Clearly, we can construct a specific copy of
$GSL_{\mathbf{0}}$, which is computable in $\mathbf{\emptyset }^{\prime
\prime }.$ Therefore, the degree of the isomorphisms type of $GSL_{\mathbf{0}%
}$ is $\mathbf{0}^{\prime \prime }.$
\end{proof}

We will now establish a similar result for the Boolean algebra $B_{\eta }.$
Recall that we identify the elements of $B_{\eta }$ with finite unions and
intersections of left-closed right-open intervals of the linear order of the
set $\mathbb{Q}$ of rationals.

\begin{theorem}
The degree of the isomorphisms type of the group $Aut_{\mathbf{0}}(B_{\eta
}) $ is $\mathbf{0}^{\prime \prime }.$
\end{theorem}

\begin{proof}
The map $H:Sym(\omega )\hookrightarrow Aut(B_{\eta })$ was defined in the
proof of part (b) of Theorem \ref{Main} in section 1. Since $H(p)$ is
uniformly computable in $p$, where $p\in Sym(\omega ),$ $H$ takes elements
of $Sym_{\mathbf{0}}(\omega )$ to elements of $Aut_{\mathbf{0}}(B_{\eta }).$
In the remainder of the proof we will be concerned with the restriction of $%
H $ to $Sym_{\mathbf{0}}(\omega ).$ We claim that under $H$, the image of a $%
1_{\inf }2_{\inf }$ permutation from $Sym_{\mathbf{0}}(\omega )$ cannot be a
conjugate of the image of a $1_{\inf }2_{\text{fin}}$ permutation from $Sym_{%
\mathbf{0}}(\omega )$.

Let $\Phi (x,\varphi )$ denote the statement%
\begin{equation*}
\left( \forall y\leq x\right) [y\neq 0\Rightarrow \left( \exists z\leq
y\right) (\varphi (z)\neq z)],
\end{equation*}%
and let $\Psi (\varphi )$ denote the statement
\begin{equation*}
\left( \exists z\right) [z=sup\{x:\Phi (x,\varphi )\}].
\end{equation*}%
The statement $\Phi $ is in the language of Boolean algebras expanded with a
function symbol $\varphi $ with an intended interpretation being an element
of $Aut_{\mathbf{0}}(B_{\eta }).$ Note that $sup$ can be defined in the
language of Boolean algebras.

Recall that in the proof of part (b) of Theorem \ref{Main}, for $p\in Sym_{%
\mathbf{0}}(\omega ),$ we defined $\widetilde{p}:\mathbb{Q\rightarrow Q}$ as
follows:%
\begin{equation*}
\widetilde{p}(x)=\left\{
\begin{array}{cc}
x & \text{if }x\ \text{is negative,} \\
p([x])+\{x\} & \text{if }x\text{ is non-negative}.%
\end{array}%
\right.
\end{equation*}

If $\varphi =H(p),$ where $p$ is a $1_{\inf }2_{fin}$ permutation on $\omega
$, then the set $A=\{x\in \mathbb{Q}:\widetilde{p}(x)\neq x\}$ is the union
of finitely many intervals of the type $[n,n+1)$ for a natural number $n$.
If $u$ is the join of these intervals in $B_{\eta }$, then $u$ is the
largest element that satisfies $\Phi (u,\varphi )$. In this case $\Psi
(\varphi )$.

If $\varphi =H(p),$ where $p$ is a $1_{\inf }2_{\inf }$ permutation on $%
\omega $, then the set $A=\{x\in \mathbb{Q}:\widetilde{p}(x)\neq x\}$ is the
union of intervals of the type $[n,n+1)$ for a natural number $n$. In this
case $\overline{A}$ is also the union of infinitely such intervals. The
formula $\Phi (u,\varphi )$ is satisfied exactly by those $u$ that are the
unions of finitely many intervals that consist entirely of elements of $A$.
Clearly, $\{u:\Phi (u,\varphi )\}$ has no supremum in this case and so $%
\urcorner \Psi (\varphi ).$

Thus, we have established that:

(1) $\Psi (\varphi )$ holds when $\varphi $ is the image, under $H,$ of an $%
1_{\inf }2_{\text{fin}}$ permutation,

(2) $\urcorner \Psi (\varphi )$ holds when $\varphi $ is the image, under $%
H, $ of an $1_{\inf }2_{\inf }$ permutation.

We will now prove that for any formula $\Theta $ in the language of Boolean
algebras, expanded with function symbols for the elements of $Aut(B_{\eta })$
and quantifiers ranging over the elements of a Boolean algebra, we have the
following for every function $f\in Aut(B_{\eta })$:%
\begin{equation*}
\Theta (\overrightarrow{x},\overrightarrow{\theta })\Leftrightarrow \Theta
(f(\overrightarrow{x}),\overrightarrow{\theta }^{f}),\text{ }
\end{equation*}%
where $\overrightarrow{\theta }=\theta _{1},\ldots ,\theta _{n}$ and $%
\overrightarrow{\theta }^{f}=\theta _{1}^{f},\ldots ,\theta _{n}^{f}$ are
new function symbols (with the usual intended interpretation). In our
notation $\Theta (\overrightarrow{x},\overrightarrow{\theta })$, $%
\overrightarrow{x}$ are the free variables and $\overrightarrow{\theta }$
are the new function symbols used in $\Theta $.

We will proceed by induction on the complexity of the formula $\Phi $.

\noindent If $\Theta (\overrightarrow{x},\overrightarrow{\theta })$ is
quantifier-free, we can consider the following atomic formulas:

$x=y,$ $x=0,$ $x=1,$ $\theta (x)=y,$ $x\wedge y=z,$ $x\vee y=z,$ $\overline{x%
}=y$. We will present only the case when $\Theta $ is the formula $\theta
(x)=y.$ Let $f\in Aut(B_{\eta }).$ Then we have:%
\begin{equation*}
\theta (x)=y\Leftrightarrow f\theta (x)=f(y)\Leftrightarrow f\theta
f^{-1}f(x)=f(y)\Leftrightarrow \theta ^{f}(f(x))=f(y).
\end{equation*}

Let $\Theta $ be the formula $\exists z\Theta _{1}(z,\overrightarrow{x},%
\overrightarrow{\theta })$ and let $f\in Aut(B_{\eta }).$

We have that $\exists z\Theta _{1}(z,\overrightarrow{x},\overrightarrow{%
\theta })\Leftrightarrow \Theta _{1}(c,\overrightarrow{x},\overrightarrow{%
\theta })$ for some $c\in B_{\eta }.$ Then, using the inductive hypothesis,%
\begin{eqnarray*}
\Theta _{1}(c,\overrightarrow{x},\overrightarrow{\theta }) &\Leftrightarrow
&\Theta _{1}(f(c),f(\overrightarrow{x}),\overrightarrow{\theta }^{f}) \\
&\Leftrightarrow &\exists z\Theta _{1}(z,f(\overrightarrow{x}),%
\overrightarrow{\theta }^{f}).
\end{eqnarray*}%
This completes the induction.

By the previous considerations, we have%
\begin{equation*}
\left( \forall f\in Aut(B_{\eta })\right) [\Psi (\varphi )\Leftrightarrow
\Psi (f^{-1}\varphi f)].
\end{equation*}%
It follows from (1) and (2) above that, under $H,$ the image of any $1_{\inf
}2_{\inf }$ permutation from $Sym_{\mathbf{0}}(\omega )$ cannot be the
conjugate of the image of any $1_{\inf }2_{\text{fin}}$ permutation from $%
Sym_{\mathbf{0}}(\omega ).$

By Theorem \ref{GSL2}, we conclude that $\mathbf{\emptyset }^{\prime \prime
} $ is computable in any copy of $Aut_{\mathbf{0}}(B_{\eta }).$ Clearly, we
can construct a specific copy of $Aut_{\mathbf{0}}(B_{\eta })$, which is
computable in $\mathbf{\emptyset }^{\prime \prime }.$ Therefore, the degree
of the isomorphisms type of $Aut_{\mathbf{0}}(B_{\eta })$ is $\mathbf{0}%
^{\prime \prime }.$
\end{proof}

Note that the results of the previous theorems in this section can be easily
relativized to any Turing degree $\mathbf{d}$.

\begin{theorem}
\label{Main2}The degree of the isomorphisms type of each of the groups $GSL_{%
\mathbf{d}}$ and $Aut_{\mathbf{d}}(B_{\eta })$ is $\mathbf{d}^{\prime \prime
}.$
\end{theorem}


\end{document}